\documentclass[a4paper]{amsart}
\usepackage{amsmath,amsthm,amssymb,amsfonts,mathrsfs,mathtools,color,hyperref, mathtools,crop, graphicx, enumitem}
\usepackage[top=3.5cm, bottom=3cm, left=2.6cm, right = 2.6cm, marginparwidth = 2cm]{geometry}
\usepackage{todonotes}
\usepackage{color}

\theoremstyle{plain}
\begingroup
\newtheorem*{maintheorem}{Main Theorem}
\newtheorem{theorem}{Theorem}[section]
\newtheorem*{theorem*}{Theorem}
\newtheorem{corollary}[theorem]{Corollary}

\newtheorem{lemma}[theorem]{Lemma}
\endgroup

\theoremstyle{definition}
\begingroup
\newtheorem{definition}[theorem]{Definition}
\endgroup

\theoremstyle{remark}
\begingroup
\newtheorem{remark}[theorem]{Remark}
\newtheorem{example}[theorem]{Example}
\endgroup 

\numberwithin{equation}{section}
\newcommand{\N}{\mathbb N} 
 
\newcommand{\R}{\mathbb R} 

\newcommand{\dist}{\operatorname{dist}}
\newcommand{\diam}{\operatorname{diam}}

\newcommand{\wto}{\stackrel*\rightharpoonup}
\newcommand{\ol}{\overline}
\newcommand{\E}{{\mathcal E}}
\newcommand{\W}{{\mathcal W}}

\renewcommand{\L}{{\mathcal L}}

\newcommand{\LRa} {\Leftrightarrow}
\newcommand{\Ra} {\Rightarrow}
\newcommand{\sdist}{\operatorname{sdist}}
\renewcommand{\H}{{\mathcal H}}
\newcommand{\spt}{\operatorname{spt}}
\newcommand{\inv}{^{-1}}
\newcommand{\cc}{\Subset}
\newcommand{\embeds}{\xhookrightarrow{\quad}}
 
\newcommand{\eps}{\varepsilon}

\newcommand{\dx}{\,\mathrm{d}x}
\renewcommand{\d}{\,\mathrm{d}}
 \newcommand{\dy}{\,\mathrm{d}y}

\begin{document}

\title[Convergence of Phase-Fields]{Uniform Regularity and Convergence of Phase-Fields for Willmore's Energy}

\author{Patrick W.~Dondl}
\address{Patrick W.~Dondl\\
Abteilung f\"ur Angewandte Mathematik\\
Albert-Ludwigs-Universit\"at Freiburg\\
Hermann-Herder-Str.~10\\
79104 Freiburg i.~Br.\\
Germany \\
Phone: +49 761 203-5642\\
Fax: +49 761 203-5644}
\email{patrick.dondl@mathematik.uni-freiburg.de}

\author{Stephan Wojtowytsch}
\address{Stephan Wojtowytsch\\Department of Mathematical Sciences\\Durham University\\Durham DH1\,1PT, United Kingdom}
\email{s.j.wojtowytsch@durham.ac.uk}

\date{\today}


\begin{abstract}
We investigate the convergence of phase fields for the Willmore problem away from the support of a limiting measure $\mu$. For this purpose, we introduce a suitable notion of essentially uniform convergence. This mode of convergence is a natural generalisation of uniform convergence that precisely describes the convergence of phase fields in three dimensions. 

More in detail, we show that, in three space dimensions, points close to which the phase fields stay bounded away from a pure phase lie either in the support of the limiting mass measure $\mu$ or contribute a positive amount to the limiting Willmore energy. Thus there can only be finitely many such points.

As an application, we investigate the Hausdorff limit of level sets of sequences of phase fields with bounded energy. We also obtain results on boundedness and $L^p$-convergence of phase fields and convergence from outside the interval between the wells of a double-well potential. For minimisers of suitable energy functionals, we deduce uniform convergence of the phase fields from essentially uniform convergence.
\end{abstract}

\maketitle

\tableofcontents

\section{Introduction and Main Results}

Problems both in pure and applied mathematics lead to investigating energies depending on the curvatures of manifolds. A famous example is Willmore's energy
\[
\W(\Sigma) = \int_\Sigma H^2\d\H^{n-1}
\]
where $\Sigma\subset\R^{n}$ is a hypersurface, $H$ denotes its mean curvature and $\H^k$ the $k$-dimensional Hausdorff measure. Often, one may reserve the term Willmore's energy for the case $n=3$ and denote the same energy on plane curves as Euler's elastica energy. Like the minimal surface problem in the gradient-theory of phase transitions \cite{modica:1987us}, this energy can be approximated by diffuse functionals on phase-fields \cite{bretin2013phase}. For a more extensive review of the literature, see e.g.\ \cite{MR3590663}. 

There are several phase-field models for Willmore's energy. The one described in the following is particularly satisfying for two reasons: 1) Higher order terms appear only quadratically, allowing for a comparatively easy and stable implementation of the gradient flow and 2) the Willmore integrand is the density associated to the first variation of the diffuse area integrand squared. In this sense, the diffuse functional is constructed in direct analogy to the sharp interface energy. 

Assume that $n=2,3$, $\Omega\cc\R^n$, $W$ is the double-well potential $W(u) = \frac14\,(u^2-1)^2$ and $c_0 = \int_{-1}^1\sqrt{2\,W(s)}\:ds = 2\sqrt{2}/3$. Then we consider the Modica-Mortola energy
\begin{equation*}
 S_\eps\colon L^1(\Omega)\to \R, \quad S_\eps(u) = \begin{cases}\frac1{c_0}\int_\Omega \frac\eps2\, |\nabla u|^2 + \frac1\eps\,W(u)\dx &u\in W^{1,2}(\Omega)\\ +\infty&\text{else}\end{cases}
\end{equation*}
as an approximation of the surface area measure and the functional
\begin{equation*}
\W_\eps\colon L^1(\Omega)\to\R, \quad \W_\eps(u) = \begin{cases}\frac1{c_0}\int_\Omega\frac1\eps\,\left(\eps\,\Delta u - \frac1\eps\,W'(u)\right)^2\dx&u\in W^{2,2}(\Omega)\\ +\infty &\text{else}\end{cases}
\end{equation*}
as an approximation of Willmore's energy. In \cite{roger:2006ta}, R\"oger and Sch\"atzle proved that 
\begin{equation*}
\left.\Gamma(L^1(\Omega))-\lim_{\eps\to 0}\,(\W_\eps + \Lambda\,S_\eps)\right.\,(\chi_E - \chi_{\Omega\setminus E})\: =\: \W(\partial E) + \Lambda\,\H^{n-1}(\partial E)
\end{equation*}
for any $\Lambda>0$ if $E\cc \Omega$ and $\partial E\in C^2$ in $n=2,3$ space dimensions. We shall always restrict ourselves to these dimensions in the following. The recovery sequence is constructed by using the well-known transition layer for the Modica-Mortola energy \cite{bellettini:1993vg}. It is given by functions of the type
\[
u^\eps(x) = q^\eps( \sdist(x,\partial E)/\eps)
\]
where $\sdist$ is the signed distance function to $\partial E$ and $q^\eps$ is a suitably cut-off approximation of $q (t) = \tanh(t/\sqrt{2})$ which is the stationary transition between the wells $u = \pm 1$ of $W$ in one dimension since $-q'' + W'(q) = 0$. Sequences of this general structure will be referred to as {\it optimal profiles}. 

Clearly, optimal profiles converge to $\chi_E-\chi_{\Omega\setminus E}$ in $L^1(\Omega)$. More generally, they converge in all finite $L^p$-spaces and they become constant away from $\partial E$. Our aim in this article is to understand the convergence properties of general sequences $u_\eps$ such that the total energy $\E_\eps(u_\eps)\coloneqq  (\W_\eps + S_\eps)(u_\eps)$ remains bounded in $n=3$ dimensions. 

We notice that such a sequence must have a limit point in the $L^2$-sense (in any dimension) since a uniform bound on $S_\eps$ implies a uniform bound on $u_\eps$ in $L^4(\Omega)$ and on $G(u_\eps)$ in $BV(\Omega)$ where $G$ is a primitive function of $\sqrt{2W}$. This argument only uses the strict monotone growth of $G$ and Young's inequality.

While $u_\eps\in W^{2,2}(\Omega)$ is continuous on $\Omega$ (and $\overline\Omega$ if $\partial \Omega \in C^{0,1}$) due to the Sobolev embedding theorems, the limit $u$ is either discontinuous or constant $\pm 1$, and uniform convergence $u_\eps\to u$ cannot be expected. 

To properly understand the convergence of a fixed finite energy sequence $u_\eps$, one needs to consider the associated Radon measures $\mu_\eps, \alpha_\eps$ given on open sets $U\subset\R^n$ by
\begin{equation*}
\mu_\eps(U) = \frac1{c_0}\int_{U\cap\Omega} \frac\eps2\, |\nabla u_\eps|^2 + \frac1\eps\,W(u_\eps)\dx, \qquad \alpha_\eps(U)= \frac1{c_0}\int_{U\cap\Omega}\frac1\eps\,\left(\eps\,\Delta u_\eps - \frac1\eps\,W'(u_\eps)\right)^2\dx
\end{equation*}
which localise the functionals $S_\eps$ and $\W_\eps$ respectively. We also denote the Willmore integrand by $v_\eps\coloneqq  -\eps\,\Delta u_\eps + \frac1\eps W'(u_\eps)$. By the compactness theorem for Radon measures, there exist finite Radon measures $\mu,\alpha$ supported in $\overline\Omega$ with $|D u|\leq 2\mu$ such that (for a further subsequence) $\mu_\eps\wto \mu$ and $\alpha_\eps\wto\alpha$. Due to \cite{roger:2006ta}, the restriction $\mu|_\Omega$ is the mass measure of an integral varifold with square integrable mean curvature and $H^2_\mu\cdot\mu|_\Omega\leq \alpha$. 
In this article, we will always make the following non-restrictive assumptions:

\begin{enumerate}
\item The sequence $u_\eps$ has finite energy, i.e.\ $\limsup_{\eps\to \infty}\E_\eps(u_\eps)<\infty$,

\item all quantities have a limit, i.e.\ $u_\eps\to u$ in $L^2(\Omega)$, $\mu_\eps\wto \mu$ and $\alpha_\eps\wto \alpha$,

\item $\eps$ is small enough for the phase fields to resemble the limit in the sense that we assume that $\mu(\R^n) = \bar\mu$, $\alpha(\R^n) = \bar\alpha$ and $\mu_\eps(\R^n)\leq \bar\mu+1$, $\alpha_\eps(\R^n) \leq\bar\alpha+1$ and

\item on $\partial\Omega$ we have either $|u_\eps|\leq \theta$ for all $\eps>0$ and some $\theta\geq1$ or $\partial_\nu u_\eps = 0$ for all $\eps>0$.
\end{enumerate}
Sequences along which the Modica-Mortola energy $S_\eps(u_\eps)$ stays bounded can behave quite badly. In specific situations, uniform convergence $u_\eps\to u$ away from the support of the measure $\mu$ has been established: 

\begin{enumerate}
\item If the functions $u_\eps$ are not only bounded in energy but also local minimisers of $S_\eps$ under the constraint $\frac1{|\Omega|}\int_\Omega u\dx = m\in(0,1)$, \cite{caffarelli:1995kh}. 
\item More generally, if they are stationary states 
\[
u_k\in W^{1,2}(\Omega), \qquad -\eps_k \,\Delta u_k + \frac1{\eps_k}\,W'(u_k) = \lambda_k
\]
with $\eps_k\to 0$ for which the sequence of Lagrange multipliers $\lambda_k$ remains bounded \cite{hutchinson:2000df}.
\item Assuming the same energy bounds as we do, but in dimension $n=2$ \cite{nagase2007singular}.
\end{enumerate}

We can take a continuous representative of $u_\eps\in W^{2,2}(\Omega)\embeds C^{0,1/2}(\overline \Omega)$ if $\Omega$ is regular and $u_\eps \in C^{0,1/2}_{loc}(\Omega)$ else. For $u$ we take the representative that is constant $\pm 1$ on $\Omega\setminus\spt(\mu)$ (which exists since $|Du|\leq 2\mu$). We will show that the following hold.

\begin{maintheorem} \label{thm:main1}
Let $n=3$, $\Omega\cc\R^n$ open and the conditions above be satisfied. Then the following hold true.
\begin{enumerate}
\item Let $\Omega'\cc\Omega$. Then there exists $C>0$ such that $|u_\eps|\leq C$ on $\Omega'$ for all $\eps < \dist(\Omega', \partial \Omega)^2$ and $u_\eps\in C^{0,1/2}(B_\eps(x))$ for all $x\in \Omega'$ with
\[
|u_\eps(y) - u_\eps(z)|\leq \frac{C}{\eps^{1/2}} \,|y-z|^{1/2}\qquad\forall\ y, z\in B_\eps(x).
\]
\item Let $\Omega'\cc\Omega$. Then $u_\eps\to u$ in $L^p(\Omega')$ for all $1\leq p<\infty$.

\item Let $\tau>0$. Then there are only finitely many points $x\in\Omega$ with the following property:
\[
\exists\ x_\eps\to x\quad\text{such that }\limsup_{\eps\to 0}|u_\eps(x_\eps)|\geq 1+\tau.
\]
The number of points can be bounded in terms of $\bar\mu$, $\bar\alpha$ and $\tau$.

\item Let $\tau>0$. Then there are only finitely many points $x\in\Omega\setminus\spt(\mu)$ with the following property:
\[
\exists\ x_\eps\to x\quad\text{such that } \limsup_{\eps\to 0} \big|\,u_\eps(x_\eps)- u(x)\big| \geq \tau.
\]
The number of such points can be bounded in terms of $\bar\mu$, $\bar\alpha$ and $\tau$. 

\item Let $\Omega'\cc\Omega\setminus\spt(\mu)$. If $\alpha$ has no atoms in $\overline{\Omega'}\setminus\spt(\mu)$, then $u_\eps\to u$ uniformly on $\Omega'$. In particular, if $V$ is an integral varifold supported in $\Omega$ with mass measure $\mu$ such that $\mu_\eps\to \mu$ and additionally $\alpha_\eps(\Omega)\to \W(\mu)$, then $u_\eps$ converges to $u$ uniformly on all $\Omega'\cc \Omega\setminus\spt(\mu)$.

\item Let $I\cc (-1,1)$. Then there exists a compact set $K\subset \overline\Omega$ and a subsequence $\eps\to 0$ such that $u_\eps^{-1}(I)\to K$ in Hausdorff distance. $K$ satisfies 
\[
K\cap\Omega = \left(\spt(\mu)\cap\Omega\right)\cup\{x_1,\dots, x_N\}
\]
 for finitely many points $x_1,\dots,x_N\in\Omega$. If $\alpha$ has no atoms outside $\spt(\mu)$, then $K\cap\Omega = \spt(\mu)\cap\Omega$.

\item There exists a countable set $\Delta\subset \Omega\setminus\spt(\mu)$, such that $u_\eps\to u$ pointwise everywhere on $\Omega\setminus(\spt(\mu)\cup \Delta)$. In particular, for $C \cc\Omega\setminus\spt(\mu)$, $s>0$ such that $\H^s(C) < \infty$ we have that $u_\eps\to u$ $\H^s|_C$-almost everywhere.  

\end{enumerate}
\end{maintheorem}

The statement and the proof are split over Corollary \ref{corollary Hs}, Lemma \ref{regularity lemma}, Theorems \ref{theorem outside n=3}, \ref{theorem away n=3} and \ref{corollary hausdorff}.

Under the same assumptions, Nagase and Tonegawa \cite{nagase2007singular} proved uniform convergence in $n=2$ dimensions. 
The differences between the cases $n=2$ and $n=3$ arise from the sharp interface problem, not the phase field approximation. Namely, due to the fact that Willmore's energy is scale invariant, the sequence of manifolds
\[
M_k = \partial B_1(0) \cup \partial B_{1/k}(0)
\]
has Willmore energy $\W(M_k) \equiv 32\pi$ in $n=3$ dimensions. It satisfies $M_k\to \partial B_1(0)$ in the measure sense, but $M_k \to \partial B_1(0) \cup \{0\}$ in Hausdorff distance. Such a sequence can be used to show that uniform convergence cannot hold for the phase field problem. The analogue of Willmore's energy on curves, also known as Euler's elastica energy, is not scale invariant since the exponent of the mean curvature $p=2$ is higher than the dimension $n-1 = 1$ of the manifold.

It is an important feature of our analysis that we only assume that $\E_\eps(u_\eps)$ is bounded and not necessarily that $u_\eps$ is a local minimiser or stationary point of a related functional under suitable side conditions. This is of central importance for applications, where Willmore's energy is usually not the only term contributing to the total energy in a model.

We remark that our result is sharp. While the formulation is new, it is geometrically intuitive. Namely, the sets
\begin{equation*}\label{equation Delta tau}
\Delta_\tau\coloneqq  \{x\in \Omega\setminus\spt(\mu)\:|\:\exists\ x_{\eps}\to x\text{ such that } \limsup_{\eps\to0}|u_{\eps}(x_\eps)-u(x)|\geq \tau\} 
\end{equation*}
and $\Delta\coloneqq  \bigcup_{\tau>0}\Delta_\tau = \bigcup_{k=1}^\infty \Delta_{1/k}$ encode how far $u_\eps$ is from converging uniformly to $u$. Since $u$ is locally constant on $\Omega\setminus\spt(\mu)$, it is easy to see that $u_\eps\to u$ locally uniformly on $\Omega\setminus\spt(\mu)$ if and only if $\Delta=\emptyset$. We show that the $\tau$-distant sets $\Delta_\tau$ are finite for all $\tau>0$, but may be non-empty. So while uniform convergence cannot be achieved in general, the set where it fails by any given positive amount is as small as can be.

This is still a strong statement, and we shall call such functions converging essentially uniformly on $\Omega\setminus\spt(\mu)$. Essentially uniform convergence is especially suited for investigating functionals that depend on individual level sets and can be used to deduce uniform convergence for certain minimising sequences, see Section \ref{section applications}. The new technique is particularly useful in fourth order problems where energy competitors cannot be constructed as easily as in generalised Modica-Mortola functionals.

Our convergence results can be applied to functionals depending on level sets of phase-fields (like the connectedness constraint functionals in \cite{MR3590663}), to rigorously justify the usual notion that the zero level set of a phase field approximates a sharp interface limit (via Hausdorff convergence) and to prove the stability of closedness under varifold convergence in the class of embedded surfaces with uniformly bounded Willmore energy.

The article is organised as follows. In Section \ref{section auxiliary}, we collect a few helpful results and first applications to boundedness and $L^p$-convergence of $u_\eps$. We deal with essentially uniform convergence in Section \ref{section n=3} and Hausdorff convergence of the level sets of $u_\eps$ to $\spt(\mu)$ in Section \ref{section Hausdorff}. Applications to uniform convergence for minimisers and varifold geometry will be discussed in Section \ref{section applications}. We conclude the article with examples demonstrating that our results are sharp in Section \ref{section examples}.

\section{Proofs} 

\subsection{Auxiliary Estimates}\label{section auxiliary}

In this section, we will collect a few estimates. The first Lemma is essentially obvious from the energy estimates, but important in controlling the Sobolev norms of $u_\eps$ from the control over $\E_\eps(u_\eps) = (\W_\eps + S_\eps)(u_\eps)$.

\begin{lemma}\label{lemma trivial regularity}
Let $u_\eps\in W^{2,2}(\Omega)$. Then there is a constant $C$ depending on $\E_\eps(u_\eps)$ and $\Omega$ such that
\[
||u_\eps||_{2,\Omega}\leq C,\qquad ||\nabla u_\eps||_{2,\Omega}\leq \frac C{\sqrt\eps}, \qquad ||\Delta u_\eps||_{2,\Omega}\leq \frac{C}{\eps^{7/2}}.
\]
Due to Sobolev embeddings, $u_\eps\in C^{0,1/2}_{loc}(\Omega)$.
\end{lemma}

We further present the following estimate of how large certain integrals of the phase-fields $u_\eps$ outside $[-1,1]$ can be.

\begin{lemma}\cite[Proposition 3.5]{roger:2006ta}\label{RS_prop_bdry}
For $n=2,3$, $\Omega \subseteq \R^n$, $\eps>0$, $u_\eps\in C^2(\Omega)$, $v_\eps \in C^0(\Omega)$, 
\[
-\eps \Delta u_\eps + \frac{1}{\eps} W'(u_\eps) = v_\eps \quad \text{in $\Omega$},
\]
and $\Omega'\cc \Omega$, $0 < r < \operatorname{dist}(\Omega', \partial\Omega)$, we have
\[
\int_{\{|u_\eps| \ge 1\}\cap \Omega'} W'(u_\eps)^2 \le C_k(1 + r^{-2k}\eps^{2k})\eps^2\int_\Omega v_\eps^2+ C_k r^{-2k}\eps^{2k} \int_{\{|u_\eps| \ge 1\}\cap \Omega} W'(u_\eps)^2
\] 
for all $k\in \N_0$.
\end{lemma}

Together, these statements imply the following.

\begin{corollary}\label{lemma quadratic decay}
Let $\Omega'\cc\Omega$. Then
\[
\frac1{\eps^3}\int_{\Omega'\cap \{|u_\eps|>1\}} W'(u_\eps)^2\dx \leq C
\]
and 
\[
\limsup_{\eps\to0} \frac1{\eps^3}\int_{\Omega'\cap \{|u_\eps|>1\}} W'(u_\eps)^2\dx \leq \alpha_\eps(\Omega).
\]
\end{corollary}

A key tool in our argument is a simplified monotonicity formula. Notably, it contains the Radon measures given by
\[
\xi_{\eps,+}(U) \coloneqq  \int_{U\cap\Omega} \left(\frac\eps2\,|\nabla u_\eps|^2 - \frac1\eps\,W(u_\eps)\right)_+\dx
\]
on open sets $U\subset\R^n$. These and their absolute value versions
\[
|\xi_{\eps}|(U) \coloneqq  \int_{U\cap\Omega} \left|\frac\eps2\,|\nabla u_\eps|^2 - \frac1\eps\,W(u_\eps)\right|\dx
\]
 are also known as the discrepancy measures because they control the deviation from the natural situation of equipartition of energy in the Modica-Mortola functional. Intuitively, a blow up of phase fields resembles solutions to the stationary Allen-Cahn equation
 \[
 -\Delta u + W'(u) = 0
 \]
 on $\R^n$ and due to a classical estimate of Modica \cite{MR803255}, these satisfy $|\nabla u|^2 \leq W(u)$, so that at least $\xi_{\eps, +}$ may be expected to vanish. For an optimal profile type sequence, $|\xi_{\eps}|(\R^n)\to 0$ exponentially fast in $\eps$, and  due to \cite[Proposition 4.9]{roger:2006ta} the same is true for general finite energy sequences. The positive part $\xi_{\eps,+}$ can be estimated quantitatively using a small parameter $\delta>0$ as follows.

\begin{lemma}\cite[Lemma 3.1]{roger:2006ta}\label{roger's lemma}
Let $n=2,3$ and assume that $u_\eps\in W^{2,2}(B_1(0))$. Then there are $C, \delta_0>0$ and $M\gg 1$ such that for all $0<\delta\leq \delta_0$, $0<\eps\leq \rho$ and 
\[
\rho_0\coloneqq  \max\{2, 1+ \delta^{-M}\eps\}\,\rho
\]
we have
\[
\rho^{1-n}\xi_{\eps,+}(B_\rho) \leq C\,\left\{\delta\,\rho^{1-n}\,\mu_\eps(B_{2\rho}) + \eps^2\delta^{-M}\rho^{1-n}\,\left(\alpha_\eps(B_{\rho_0}) + \int_{B_{\rho_0}\cap\{|u_\eps|>1\}}\frac1{\eps^3}\,W'(u_\eps)\dx\right) + \frac{\eps\delta}\rho\right\}.
\]
\end{lemma}

 With this in mind, we may state our estimated monotonicity formula. A similar estimate is also given in \cite[Proposition 4.3]{roger:2006ta}; see also \cite[Lemma 3.6]{MR3590663}. 
 
\begin{lemma}\label{lemma estimated monotonicity formula}
Let $0<r<R\leq 1$, $x\in \Omega$ such that $B_R(x)\cc\Omega$. Denote $B_\rho\coloneqq  B_\rho(x)$. Then
\[
r^{1-n}\mu_\eps(B_r)\leq 3\,R^{1-n}\,\mu_\eps(B_R) + 3\,\alpha_\eps(B_R) + 2\int_r^R\frac{\xi_{\eps,+}(B_\rho)}{\rho^n}\d\rho.
\]
\end{lemma}

In the case that $u_\eps$ satisfies the boundary conditions \eqref{equation boundary conditions}, the following Lemma was already proved in \cite[Lemma 3.1]{MR3590663}, but the proof also goes through in our case. Set
\[
\Omega_\eps:= \{x\in \Omega\:|\: B_{2\eps}(x)\subset\Omega\}.
\]

\begin{lemma}\label{regularity lemma}
There is $C_{\bar\alpha, \bar\mu, n}>0$ such that 
\[
||u_\eps||_{\infty, \Omega_\eps}\leq C_{\bar\alpha, \bar\mu, n}.
\]
Take $x\in \Omega_\eps$ and set $B_\eps\coloneqq  B_\eps(x)$. Then $u_\eps$ is H\"older-continuous on $B_\eps$ with
\[
|u_\eps(y) - u_\eps(z)|\leq \frac{C_{\bar\alpha,\bar\mu, n, \gamma}}{\eps^\gamma}\,|y-z|^\gamma
\]
for all $y,z\in B_\eps$ and $\gamma\leq 1/2$ if $n=3$, $\gamma<1$ if $n=2$. 
\end{lemma}

Optimal interfaces have precisely these H\"older-coefficients, so they cannot be improved. With Corollary \ref{lemma quadratic decay}, the regularity Lemma could be extended to the case without boundary conditions on compactly contained sets.

\begin{remark}
Note that $\Omega_\eps$ is growing as $\eps\to 0$, so that the local boundedness and H\"older continuity hold on every set $\Omega'\cc\Omega$ with constants independent of $\Omega'$, at least for small enough $\eps>0$. We shall make use of this in the following. 
\end{remark}

\begin{corollary}
Let $1\leq p<\infty$. Then $u_\eps\to u$ in $L^p(\Omega')$ for all $\Omega'\cc\Omega$.

\end{corollary}

\begin{proof}
We know that $u_\eps\to u$ in $L^1(\Omega)$ and that the sequence $u_\eps$ is bounded uniformly in $L^\infty(\Omega')$. H\"older's inequality does the rest.
\end{proof}

\subsection{Essentially Uniform Convergence} \label{section n=3}

In this section, we will investigate the convergence of $u_\eps$ in $n=3$ dimensions. As we shall see in Example \ref{introductory example}, uniform convergence away from the interface does not hold in this case. Therefore, we are forced to introduce a new notion of convergence which is better adapted to phase field problems.

\begin{definition}
Let $U\subset\R^n$, $f_\eps, f\colon U\to\R$ continuous functions. Then we say that $f_\eps\to f$ {\it essentially uniformly (e.u.)} if the sets
\[
\Delta_\tau\coloneqq  \{x\in U\:|\:\exists\ x_{\eps}\to x\text{ such that } \limsup_{\eps\to0}|f_{\eps}(x_\eps)-f(x)|\geq \tau\} 
\]
are finite for all $\tau>0$.
\end{definition}

Since we assume $f$ to be continuous, locally uniform convergence corresponds to $\Delta_\tau = \emptyset$ for all $\tau>0$ and implies essentially uniform convergence. Even without the assumption of continuity, e.u.\ convergence implies convergence pointwise everywhere on the complement of a countable set. 
With this definition, our results on convergence in three dimensions can be summarised as
\[
u_\eps\to u \text{ e.u.\ on }\Omega\setminus\spt(\mu)\qquad \text{and}\qquad \left(|u_\eps|-1\right)_+\to 0\text{ e.u.\ on } \Omega.
\]

\begin{remark} \label{remark ess unif}
Essentially uniform convergence is a powerful tool for our purposes, but still quite far from uniform convergence. The following properties are easy to establish.

\begin{enumerate}
\item Assume that $f_\eps\to f$ e.u.~on $U$. Then $\Delta= \bigcup_{\tau>0} \Delta\tau = \bigcup_{k=1}^\infty \Delta_{1/k}$ is countable and $f_\eps(x)\to f(x)$ for all $x\in U\setminus \Delta$.

\item Let $K\cc U\setminus \Delta$. Then $f_\eps\to f$ uniformly on $K$.

\item $\Delta$ is countable and may lie dense in $U$, in which case the previous point is vacuous. In particular, it may happen that $f_\eps\to f$ e.u.\ but there exists no open set $U' \subset U$ such that $f_\eps\to f$ uniformly on $U'$. We shall see in Example \ref{introductory example} that this may happen in our case of finite energy sequences $u_\eps$.
\end{enumerate}
\end{remark}
In one space dimension, the same kind of convergence was used by Dal Maso and Iurlano for phase-fields governed by a Modica-Mortola energy \cite[Proof of Proposition 1]{MR2997535}. In one dimension, the Modica-Mortola functional controls functions well enough to show essentially uniform convergence. In Remark \ref{remark generalise} we discuss under what assumptions our techniques can be adapted to prove essentially uniform convergence in higher dimensions.  

We begin by proving convergence from outside $[-1,1]$, also at $\spt(\mu)$. 

\begin{theorem}\label{theorem outside n=3}
Let $\tau>0$ and $x\in \Omega$ a point for which there exists a sequence $x_\eps\to x$ such that $\limsup_{\eps\to 0} |u_\eps(x_\eps)|\geq 1+\tau$. Then there exists $\bar\theta>0$ depending only on $\bar\alpha, \bar\mu$ and $\tau$ such that $\alpha(\{x\})\geq \bar\theta$. In particular, there are only finitely many such points.
\end{theorem}

\begin{proof}
Passing to a subsequence and replacing $\tau$ by $\tau/2$, we may assume that $|u_\eps(x_\eps)|\geq 1+\tau$ for all $\eps$. Since $\Omega$ is open, there exists $r>0$ such that $B_{4r}(x) \subset\Omega$. Thus $B_{3r}(x)\subset\Omega_\eps$ for all sufficiently small $\eps$, so we may use Lemma \ref{regularity lemma} with uniform constants.
Since $x_\eps\to x$, for all sufficiently small $\eps>0$ we have $B_\eps(x_\eps)\subset B_r(x)$, and by H\"older-continuity of $u_\eps$, there is $0<c<1$ such that
\[
|u_\eps|\geq 1 + \frac{\tau}2 \quad\text{on }B_{c\eps}(x_\eps)
\]
which implies that
\[
\mu_\eps(B_r(x))\geq \mu_\eps(B_{c\eps}(x_\eps))\geq \omega_n\,(c\eps)^n\,\frac{W(1+\tau/2)}\eps.
\]
Using Corollary \ref{lemma quadratic decay}, we find that
$\alpha(B_{2r}) \geq \omega_n\,c^n\,W(1+\tau/2)$ where $c$ only depends on the H\"older constant of $u_\eps$ on $B_{c\eps}(x_\eps)$ and thus only on the energy bounds. Taking $r\to 0$, we see that
\[
\alpha(\{x\}) \geq C_{\bar\alpha,\bar \mu, \tau}.
\]
A point with the properties of $x$ is therefore an atom of $\alpha$ with a minimal size depending on $\bar\alpha, \bar\mu$ and $\tau$. In particular, since $\bar\alpha<\infty$, there are only finitely many such points.
\end{proof}

Note that we had to use the limiting measure $\alpha$. Its existence may always be achieved by taking a subsequence $\eps\to 0$. On the other hand, if we add bumps as Example \ref{introductory example} based at points along a dense sequence in some $\Omega'\cc\Omega\setminus\spt(\mu)$, we see that all points $x\in \Omega'$ are limits of bad sequences. Thus the existence of $\alpha$ is of critical importance for the argument above.

The following result is the three dimensional version of \cite[Lemma 3.9]{MR3590663}, obtained with a slightly improved estimate and streamlined argument. We slightly abuse notation and denote by $\W_\eps, S_\eps, \E_\eps$ also the functionals given by the same formulae as above on the function space $L^1(B_1(0))$ instead of $L^1(\Omega)$.

\begin{lemma}\label{minimisation lemma}
Let $n=2,3$, $B= B_1(0)\subset\R^n$, $\theta\in[0,1)$ and 

\[X_\theta\coloneqq  \{u\in W^{2,2}(B)\:|\: |u(0)|\leq \theta\}.\]
Then the function 

\[e\colon [0,1)\to \R, \quad e(\theta) \coloneqq  \liminf_{\eps\to 0} \inf_{u\in X_\theta}\E_\eps(u)\]
is strictly positive.
\end{lemma}

\begin{proof}
For a contradiction, assume that there is $\theta\in[0,1)$ and a sequence $u_\eps\in X_\theta$ such that $\E_\eps(u_\eps)\to 0$. As usual, denote $B_\rho\coloneqq  B_\rho(0)$ and the diffuse mass and Willmore measures by $\mu_\eps$ and $\alpha_\eps$, respectively, despite the change of domain. Consider the densities
\[
f_\eps(\rho)\coloneqq  \rho^{1-n}\mu_\eps(B_\rho)
\]
for $\rho\in[\eps, 1]$. By the H\"older continuity on $B_{1/2}$ from Lemma \ref{regularity lemma}, we get $f_\eps(\eps) = \eps^{1-n}\mu_\eps(B_\eps)\geq \bar c>0$ for a uniform constant depending only on $\theta$ by the same argument as in Theorem \ref{theorem outside n=3} (since $\bar\mu = \bar\alpha=0$ by assumption). In the next step, we will apply Lemma \ref{roger's lemma} with $\delta = \eta_\eps\,(\eps/\rho)^\beta$ for some $0<\beta < 1/M$ and $\eta_\eps\to 0$ so slowly that 

\begin{enumerate}
\item $\eta_\eps^{-M}\alpha_\eps(B)\to 0$ and

\item $\eta_\eps^{-M}\eps^{1-M\beta}\leq 1$.
\end{enumerate}

Note that the second condition also implies that $\delta^{-M}\eps = (\eps/\rho)^{-M\beta}\,\eta_\eps^{-M}\eps\leq 1$ for $\rho\geq\eps$. In particular, $\delta<\delta_0$ independently of $\rho\geq\eps$ for all small enough $\eps>0$. Using the estimated monotonicity formula from Lemma \ref{lemma estimated monotonicity formula} for $\eps = r< R = 1/3$ together with the estimates for 

\begin{itemize}
\item[--] $\xi_{\eps,+}$ from Lemma \ref{roger's lemma} for the $\delta$ given above, for
\item[--] $||u_\eps||_{\infty, B_{2/3}}$ from Lemma \ref{regularity lemma} and for 
\item[--] $\int_{B_{2/3}\cap\{|u_\eps|>1\}}\frac1{\eps^3}W'(u_\eps)^2\dx$ as in Corollary \ref{lemma quadratic decay},
\end{itemize}

we obtain
\begin{align*}
f_\eps(\eps) &\leq 3\,R^{1-n}\,\mu_\eps(B_R) + 3\,\alpha_\eps(B_R) + 2\int_r^R\frac{\xi_{\eps,+}(B_\rho)}{\rho^n}\d\rho\\
	&\leq 3\,R^{1-n}\,\mu_\eps(B_R) + 3\,\alpha_\eps(B_R) + 2\,C\int_r^R \eta_\eps\,\frac{\eps^\beta}{\rho^{1+\beta}}\,\rho^{1-n}\,\mu_\eps(B_{2\rho})\d\rho\\
		&\qquad + \int_r^R\frac{\eps^{2-M\beta}}{\rho^{n-M\beta}}\,\eta_\eps^{-M} \left(\alpha_\eps(B_{2\rho}) + \int_{B_{2\rho}\cap\{|u_\eps|>1\}}\frac1{\eps^3}W'(u_\eps)^2\dx\right) + \frac{\eps^{1+\beta}\,\eta_\eps}{\rho^{2+\beta}}\d\rho\\
	&\leq 3\,R^{1-n}\,\mu_\eps(B_R) + 3\,\alpha_\eps(B_R) + \int_{r}^R\frac{2\,C\,\eta_\eps\,\eps^\beta}{\rho^{1+\beta}}\,f_\eps(2\rho)\d\rho + \frac{C}{1+\beta}\,\eps^{1+\beta}\,\left[r^{-(1+\beta)}- R^{-(1+\beta)}\right]\,\eta_\eps \\
		&\qquad + \frac{C}{n-1-M\beta}\,\eps^{2-M\beta}\,\left\{r^{1-n+M\beta} - R^{1-n+M\beta} \right\}\,\eta_\eps^{-M}\cdot \left\{\alpha_\eps(B_{2R}) + \frac{1 - \left(\frac{5\eps}{R}\right)^3}{1 - \frac{5\,\eps}{R}}\,\alpha_\eps(B_{2R}) + \right.\\
		&\hspace{9.2cm}\left. \frac1{\eps^2}\,\left(\frac{5\,\eps}R\right)^3 \,||u_\eps||_{\infty, B_{2R}}\cdot 4\,\mu_\eps(B_{2R})\right\}\\
	&\leq \gamma_\eps + \int_{r}^{2R}\frac{2\,C\,\eta_\eps\,\eps^\beta}{\rho^{1+\beta}}\,f_\eps(\rho)\d\rho
\end{align*}
with $\gamma_\eps\to 0$ as $\eps\to 0$. We may now use Gr\"onwall's inequality backwards in time to deduce that
\begin{align*}
f_\eps(\eps) &\leq \gamma_\eps\,\exp\left(\int_\eps^{2/3} \frac{C\,\eta_\eps\,\eps^\beta}{\rho^{1+\beta}}\d\rho \right) \leq C\,\gamma_\eps.
\end{align*}
This is a contradiction since $\gamma_\eps\to 0$, but on the other hand $f_\eps(\eps)\geq \bar c>0$ due to H\"older continuity. 
\end{proof}

In the next step of our program, we will reduce the problem of uniform convergence to this minimisation problem. The central tool in doing so is the following rescaling result, compare e.g.\ the proof of \cite[Theorem 5.1]{roger:2006ta}.

\begin{lemma}\label{rescaling lemma}
Let $u_\eps\colon B(x,r)\to \R$, $\lambda>0$ and $\hat u_\eps\colon B(0,r/\lambda)\to \R$ with

\[\hat u_\eps(y) = u_\eps(x+\lambda y).\]
Set $\hat r\coloneqq  r/\lambda$, $\hat \eps\coloneqq  \eps/\lambda$, 

\[\hat\mu_\eps\coloneqq  \frac1{c_0}\: \left(\frac{\hat \eps}2\,|\nabla \hat u_\eps|^2 + \frac1{\hat\eps}\,W(\hat u_\eps)\right)\L^n,\qquad \hat \alpha_\eps\coloneqq  \frac1{c_0\,\hat\eps}\,\left(\hat \eps\,\Delta \hat u_\eps - \frac1{\hat\eps}\,W'(\hat u_\eps)\right)^2\,\L^n.\]
Then 

\[\hat r^{1-n}\hat\mu_\eps(B(0,\hat r)) = r^{1-n}\,\mu_\eps(B(x,r)),\qquad \hat r^{3-n}\, \hat\alpha_\eps(B(0,\hat r)) = r^{3-n}\,\alpha_\eps(B(x,r)).\]
\end{lemma}

With this in mind, we proceed to our main result on convergence away from $\spt(\mu)$ in three dimensions. The proof resembles that of \cite[Theorem 2.1]{MR3590663}, where uniform convergence in two dimensions was established. 

\begin{theorem}\label{theorem away n=3}
Let $\tau>0$ and $x\in\Omega\setminus \spt(\mu)$ such that there exists a sequence $x_\eps\to x$ with the property that
\[
\limsup_{\eps\to 0}|u_\eps(x_\eps) - u(x) | \geq \tau.
\]
Then there exists $\bar\theta$ depending only on $\tau$ such that $\alpha(\{x\})\geq \bar\theta$. In particular, there are only finitely many such points.
\end{theorem}

\begin{proof}
Since convergence from outside $[-1,1]$ is already clear, we only consider $x$ that admit a sequence $x_\eps\to x$ such that
\[
\liminf_{\eps\to 0}|u_\eps(x_\eps)|\leq 1-\tau.
\]
Without loss of generality, we may assume that $u(x) =1$. Assume that there is a subsequence $x_\eps\to x$ such that $u_\eps(x_\eps)<0$. Since $u_\eps\to u$ in $L^1(\Omega)$ (so pointwise almost everywhere, up to a subsequence), and $u$ is locally constant, there is also a sequence $\tilde x_\eps\to x$ such that $u_\eps(\tilde x_\eps)\geq 1-\tau/2$. Using the continuity of $u_\eps$, we obtain a sequence $x_\eps'\to x$ such that $|u_\eps(x_\eps')|\leq 1-\tau$. Passing to a subsequence in $\eps$, we may assume that this holds for all $\eps$.

So assume that $x_\eps\to x\in\Omega$ and $|u_\eps(x_\eps)|\leq 1-\tau$. Since $\Omega\setminus\spt(\mu)$ is open, there is $r>0$ such that $B(x,3r)\subset\Omega\setminus\spt(\mu)$. As $x_\eps\to x$, $B(x_\eps,r)\subset B(x,2r)$ for almost all $\eps>0$. We have $\mu(B(x,3r))=0$, so (using the terminology of Lemmas \ref{minimisation lemma} and \ref{rescaling lemma})
\begin{align*}
\alpha(B_{3r}(x)) &\geq \alpha(\overline{B_{2r}(x)})\\
	&\geq \limsup_{\eps\to 0} \left(\alpha_\eps(B_{2r}(x) + r^{1-n}\mu_\eps(B_{2r}(x))\right)\\
	&\geq \limsup_{\eps\to0} \left(\alpha_\eps(B_r(x_\eps)) + r^{1-n}\mu_\eps(B_r(x_\eps))\right)\\
	&= \limsup_{\hat\eps \to 0} \left(\hat \alpha_\eps(B_1(0)) + \hat \mu_\eps(B_1(0))\right)\\
	&\geq \limsup_{\hat\eps\to 0}\inf_{u\in X_{1-\tau}} \left(\W_{\hat\eps} + S_{\hat\eps}\right)(u)\\
	&\geq\bar\theta
\end{align*}
with $\hat u_\eps(y) = u_\eps(x_\eps+ry)$ and $\hat\eps = \eps/r$. Letting $r\to 0$, we establish that
\[
\alpha(\{x\})\geq \bar\theta
\]
where $\bar\theta$ only depends on $\tau$. Again, $x$ is an atom of a fixed minimal size, so there are only finitely many such points.
\end{proof}

\begin{corollary}\label{corollary recovery sequence}
Assume that $\Omega'\cc\Omega\setminus \spt(\mu)$. Then the following hold true.

\begin{enumerate}
\item For all $\tau>0$ there exists $\bar c_\tau>0$ such that if $\alpha$ has no atoms of size at least $\bar c$ in $\overline{\Omega'}$, then $\big|\,|u_\eps|-1\big|< \tau$ on $\Omega'$. 

\item If $\alpha$ has no atoms in $\overline{\Omega'}$ at all, then $u_\eps\to u$ uniformly on $\Omega'$.

\item If $\mu$ is the mass measure of a varifold $V$ and $\alpha(\overline\Omega) = \W(V)$ (i.e.\ $u_\eps$ is a recovery sequence for its limit), then $u_\eps\to u$ locally uniformly in $\Omega\setminus\spt(\mu)$.
\end{enumerate}
\end{corollary}

\begin{proof}
All but the last point are obvious. Clearly, it suffices to show that $\alpha$ has no atoms outside $\spt(\mu)$. For a contradiction, assume that $x_0\notin \spt(\mu)$ is an atom of $\alpha$ and choose $\Omega' = \Omega\setminus \overline{B_r(x_0)}$ such that $B_r(x_0)\cc\Omega\setminus\spt(\mu)$. Then consider the sequence $\bar u_\eps = u_\eps$ pointwise. Clearly still $\bar\mu_\eps\wto\mu$, but $\liminf_{\eps\to 0}\W_\eps(\bar u_\eps) < \W(V)$ contradicting the $\Gamma-\liminf$ inequality from \cite{roger:2006ta}.
\end{proof}

The following is an easy corollary once essentially uniform convergence is established. We state it here in order to illustrate the properties of this mode of convergence.
\begin{corollary} \label{corollary Hs}
There exists a countable set $\Delta\subset \Omega\setminus\spt(\mu)$, such that $u_\eps\to u$ pointwise everywhere on $\Omega\setminus(\spt(\mu)\cup \Delta)$. In particular, for $C \cc\Omega\setminus\spt(\mu)$, $s>0$ such that $\H^s(C) < \infty$ we have that $u_\eps\to u$ $\H^s|_C$-almost everywhere. 
\end{corollary}
\begin{proof}
The statement follows from Remark~\ref{remark ess unif} point (1), which is evident from the definition of essentially uniform convergence.
\end{proof}

A few remarks are in order.

\begin{remark}\label{remark n=2,3}
The only difference to the case $n=2$ lies in the different rescaling properties of $\alpha_\eps$ in two and three dimensions. There, we could deduce that $\alpha(B_{3r}(x)) \geq \bar\theta/r$, which gives a contradiction as $r\to 0$ and establishes uniform convergence of $|u_\eps|\to 1$ on sets $\Omega'\cc\Omega\setminus\spt(\mu)$. 
\end{remark}

\begin{remark}\label{remark recovery sequence}
As pointed out, if $\Omega'\cc \Omega$ and $\mu(\overline{\Omega'}) = \alpha(\overline{\Omega'}) = 0$, then $|u_\eps|\to 1$ uniformly on every $\Omega''\cc\Omega'$. However, the convergence has no {\it a priori} rate in $\eps$ in $n=3$ dimensions. Functions like
\[
u_\eps = 1 + f(\eps)\,g(\,(x-x_0)/\eps)
\]
will not lead to atoms of $\alpha$ if $g\in C_c^\infty(\R^n)$ and $f(\eps)\to 0$ as $\eps\to 0$. For similar considerations, see Example \ref{introductory example}. In two dimensions, the optimal rate of convergence is $\sqrt{\eps}$ as shown in \cite{nagase2007singular} from outside $[-1,1]$. The argument can also be applied in general.
\end{remark}
  
\begin{remark}\label{remark generalise}
The argument presented above can clearly be adapted to other situations with the following ingredients: 

\begin{enumerate}
\item a sequence of functions $u_\eps$  which induces two sequences of Radon measures $\mu_\eps,\alpha_\eps$ uniformly bounded on compact subsets

\item Limiting objects such that $u_\eps\to u$, $\mu_\eps\wto \mu$ and $\alpha_\eps\wto\alpha$,

\item an infinitesimal generation of mass property like 
\[|u_\eps(x) - u(x)|\geq \theta \qquad\Ra\qquad \eps^{1-n}\mu_\eps(B_\eps)\geq \bar c_\theta,\]

\item a monotonicity formula resembling
\[
R^{1-n}\mu_\eps(B_R) \geq c_1\,r^{1-n}\mu_\eps(B_r) - c_2\,\alpha_\eps(B_R) + \Xi_\eps, \qquad c_1,c_2>0, \qquad \eps\leq r\leq R
\]
for $\mu_\eps$ which involves only $\mu_\eps, \alpha_\eps$ and an error term $\Xi_\eps$ which goes to zero and

\item a critical or sub-critical rescaling property for $\alpha_\eps$ and an invariance property for the `density' $r^{1-n}\mu_\eps(B_r)$.
\end{enumerate}

Then we can re-write the problem of uniform convergence into a minimisation problem and employ the same arguments as above. Depending on the nature of the rescaling property, we may be able to obtain uniform convergence this way (as for $n=2$) or essentially uniform convergence (as for $n=3$).
\end{remark}

\subsection{Hausdorff Convergence} \label{section Hausdorff}

In applications, we like to think of $\spt(\mu)$ as being approximated by the set $\{u_\eps=0\}$. This is rigorously justified in the next theorem.

\begin{theorem}\label{corollary hausdorff}
Let $I\cc(-1,1)$ be non-empty, not necessarily open. Then, up to a subsequence, $u_\eps^{-1}(I)$ converges to a compact set $K\subset\overline\Omega$ in Hausdorff distance such that

\begin{enumerate}
\item $K\cap \Omega = \spt(\mu)\cap\Omega$ if $n=2$ or $n=3$ and $\alpha$ has no atoms in $\Omega\setminus\spt(\mu)$,

\item $K\cap\Omega = \left(\spt(\mu) \cap \Omega\right)\cup \bigcup_{k=1}^N \{x_k\}$ for finitely many points $x_k\in \Omega$ if $n=3$. The number $N$ of points can be bounded in terms of $I$ and $\limsup_{\eps\to 0}\E_\eps(u_\eps)$.
\end{enumerate}
\end{theorem}

\begin{proof}
In accordance with convention, we may replace $u_\eps^{-1}(I)$ with its closure without affecting the limit. Since $u_\eps^{-1}(I) \subset\Omega$ is bounded, there is a compact set $K\subset\overline\Omega$ and a subsequence such that
\[
u_\eps^{-1}(I)\to K
\]
in Hausdorff distance. $K$ can be calculated as the Kuratowski limit
\[
K = \{x\in \overline\Omega\:|\:\exists\ x_\eps\in u_\eps^{-1}(I)\text{ such that }x_\eps\to x\}.
\]
In \cite[Theorem 2.2 and Lemma 3.15]{MR3590663}, we showed that $\spt(\mu) \subset K$ under the assumption that $u_\eps \in -1 + W^{2,2}_0(\Omega)$. Our original argument had to be more sophisticated to imply a stronger result, so we shall sketch a simplified version of the proof.

\begin{enumerate}
\item $\spt(\mu)$ is rectifiable and when we introduce the diffuse normal direction $\nu_\eps = \frac{\nabla u_\eps}{|\nabla u_\eps|}$, the varifolds $V_\eps \coloneqq  \mu_\eps \otimes \nu_\eps$ converge to the varifold introduced by $\mu$ as Radon measures on $\R^n\times G(n, n-1)$ due to \cite[Proposition 4.1]{roger:2006ta}. The discrepancy measures 
\[
|\xi_\eps|(U) = \int_U \left|\frac\eps2\,|\nabla u_\eps|^2 - \frac1\eps\,W(u_\eps)\right|\dx
\]
vanish as $|\xi_\eps|\wto 0$ \cite[Proposition 4.9]{roger:2006ta}.

\item Since points $x$ at which $\mu$ possesses a tangent space $T_x\mu$ and which are Lebesgue points of the density $\theta$ and not atoms of $\alpha$ lie dense in $\spt(\mu)$, it suffices to show the following: For $x$ as specified and $r>0$, there is $z\in u_\eps^{-1}(I) \cap B_r(x)$ for all sufficiently small $\eps>0$. 

Blowing up along a suitable sequence of dilations $\tilde u_\eps(y) = u_\eps(x+\lambda_\eps \,y)$ with $\lambda_\eps \to\infty$, $\eps\,\lambda_\eps\to 0$, we can further reduce this to the case in which $\mu = \theta\,\H^{n-1}|_{T_x\mu}$, $\alpha =0$ like in \cite[Theorem 5.1]{roger:2006ta}. 

\item Let $0<\tau<1-1/\sqrt{2}$ and $L, \gamma>0$ to be specified later. When we use this varifold convergence and $|\xi_\eps|\wto0$, for all sufficiently small $\eps>0$ we can find a point $y = y_\eps \in B_{r/2}(x)$ such that $|u_\eps(y)|\leq 1-\tau$ and (up to rotation)
\[
|\xi_\eps|(B_{4L\eps}(y)) + \int_{B_{4L\eps}(y)} 1 - \langle \nu_{\eps},e_n\rangle^2\d\mu_\eps \leq \gamma\,(4L\eps)^{n-1}.
\]
Imagining $L$ to be large and $\gamma$ small, this roughly expresses that $\nabla u_\eps$ points in the direction $e_n$ on average over $B_{4L\eps}(y)$ and that the discrepancy measure over the ball is small.

\item If we choose $L$ large enough and $\gamma$ small, then \cite[Lemma 3.12]{MR3590663} or the proof of \cite[Proposition 5.5]{roger:2006ta} show that $u_\eps$ is $C^0$-close to an optimal profile in direction $e_n$ over $B_{4L\eps}(y)$ and makes an almost full transition from $-1$ to $1$ within that ball. We thus find $z\in B_{4L\eps}(y)\subset B_r(x)$ such that $u_\eps(z)\in I$.
\end{enumerate}

For the other direction, assume that $x\in \Omega \cap K\setminus\spt(\mu)$. Take $r>0$ such that $B_{r}(x)\subset \Omega\setminus\spt(\mu)$. If $n=2$ or $n=3$ and $\alpha$ has no atoms in $\overline{B_r(x)}$, we see that $|u_\eps|\to 1$ uniformly on $\overline{B_r(x)}$, which leads to a contradiction. If $n=3$ in general, then $x$ must be an atom of $\alpha$ with a minimal size depending only on
\[
\sup_{\theta\in I}|\theta|<1.
\]
Since there can only be finitely many such points, the theorem is proven.
\end{proof}

\begin{remark}
In the case where we have $u_\eps\in -1 + W^{2,2}_0(\Omega)$, we can extend $u_\eps$ to a larger domain as a constant function outside $\Omega$. Thus we obtain the stronger result that $u_\eps^{-1}(I)\to \spt(\mu)$ if $n=2$ (or if $n=3$ and $\alpha$ does not have atoms outside $\spt(\mu)$) and $u_\eps^{-1}(I)\to \spt(\mu)\cup \{x_1,\dots,x_N\}$ (up to a subsequence) if $n=3$ for a finite collection of points $x_i\in \overline\Omega$. If $n=2$ or $n=3$ and $\alpha$ has no atoms, the uniqueness of the limit implies that actually the whole sequence $u_\eps^{-1}(I)$ converges to $\spt(\mu)$. The same holds for periodic boundary conditions.

Without boundary conditions, the relationship of $K\cap \partial\Omega$ and $\spt(\mu)\cap \partial\Omega$ is more complicated. If $\partial\Omega\in C^2$, we may consider an optimal interface transition for $\partial \Omega$ such that only the positive part of the transition lies inside $\Omega$. This induces the measure $\mu = 1/2\,\H^{n-1}|_{\partial\Omega}$. So $\mu$ may well fail to be an integral varifold at the boundary, and the inclusion $\spt(\mu)\cap\partial\Omega \not\subset K\cap \partial\Omega$ need not hold (take $I\cc (-1,0)$). Further details are given in a forthcoming publication \cite{DW_bound}.
\end{remark}

\section{Applications}\label{section applications}

\subsection{Enforcing Connectedness in Limits of Phase Fields}
Our research into the uniform convergence was motivated by difficulties in \cite{MR3590663} where we needed phase fields to converge along curves (thus objects of co-dimension $2$) also in dimension $n=3$. Accordingly, our results can be used to simplify the proof of \cite[Theorem 2.5]{MR3590663}, but there are many other applications.

\begin{remark}
In the proof of \cite[Theorem 2.5]{MR3590663}, once we have shown that (in the terminology of that Theorem) $u_\eps\in [\theta_1,\theta_2] \cc (-1,1)$ on a connected set $K_\eps$ bounded away from $\spt(\mu_\eps)$ which contains two distinct points with a uniformly positive separation, we can obtain a contradiction more directly. All points in the Hausdorff limit $K$ of $K_\eps$ must have have the distance property, and since $K$ is connected and contains two distinct points, it is not a finite collection of points. Since furthermore $K\subset\Omega\setminus\spt(\mu)$, this poses a contradiction to Theorem \ref{theorem away n=3}.
\end{remark}

\subsection{Minimising Sequences Converge Uniformly}\label{section minimisers converge uniformly}
In the second application, we demonstrate how essentially uniform convergence can be used to obtain uniform convergence under additional assumptions. It formalises the intuition that phase fields have no energetic incentive not to converge uniformly in three dimensions. 

\begin{lemma}\label{lemma uniform convergence minimisers}
Let $X= -1 + W^{2,2}_0(\Omega)$, $S, V\in\R$, $\lambda>0$, $\chi\geq 0$ and
\[
\E_\eps\colon X\to [0,\infty), \qquad \E_\eps(u) = \W_\eps(u) + \lambda\,(S_\eps(u)-S)^2 + \chi\,\left(\frac12\int_\Omega (u+1)\dx - V\right)^2
\]
an associated energy functional. Furthermore, assume that $u_\eps\in X$ and $u\in BV(\Omega)$ are such that 
\[
\E_\eps(u_\eps) = \min_{v\in X} \E_\eps(v), \qquad u_\eps\to u\text{ in }L^1(\Omega).
\]
As usual, let $\mu_\eps\wto\mu$, $\alpha_\eps\wto \alpha$. Then $\spt(\alpha)\subset \spt(\mu)$. In particular, $u_\eps\to u$ uniformly on compact sets $K\subset \overline\Omega\setminus\spt(\mu)$.
\end{lemma}


The parameters $S$ and $V$ play the roles of a preferred surface area and enclosed volume and $\lambda, \chi$ express the strength of the preference.

The existence of minimisers of $\E_\eps$ follows from the direct method of the calculus of variations and Sobolev embedding theorems. A similar statement holds if $X$ is $W^{2,2}(\Omega)$ or the subspace of $W^{2,2}(\Omega)$ with vanishing normal derivatives.

\begin{proof}
Note that the sequence $\bar u_\eps\equiv -1$ keeps $\E_\eps$ uniformly bounded, so we have $\limsup_{\eps\to 0}\E_\eps(u_\eps)<\infty$. Let us consider a subsequence $\eps\to 0$ such that all three terms in the energy have a limit.

By Corollary \ref{corollary recovery sequence}, $u_\eps\to u$ locally uniformly in $\overline\Omega\setminus\spt(\mu)$ if $\alpha$ has no atoms outside $\spt(\mu)$, so it suffices to show the inclusion $\spt(\alpha)\subset\spt(\mu)$. By extending $u_\eps$ to a slightly larger domain $\Omega'$ as a constant function, we only need to consider the case that $\spt(\alpha)\cc\Omega$. Recall that the support of a measure is the collection of all points for which any neighbourhood has positive measure. Thus for a contradiction, we may assume that there exists a ball $B_{2r}(x)\subset \Omega\setminus\spt(\mu)$ such that $\alpha(B_r(x))>0$. 

Since there are only finitely many points $x\in\overline\Omega\setminus\spt(\mu)$ such that there exists a sequence $x_\eps\to x$ with the property that $\limsup_{\eps\to 0}|\,|u_\eps(x_\eps)|-1|\geq \tau$ for any given $\tau>0$, we can choose two radii $r< r_1 < r_2 < 2r$ and the ring domain
\[
R\coloneqq  \{y\in \Omega \:|\: r_1 < |x-y|< r_2\}
\]
such that $|u_\eps|\geq 1/\sqrt{2}$ on $\overline{R'}$ for all sufficiently small $\eps$ and a slightly larger set $R'$ such that $R\cc R'$. We now need a statement similar to Corollary \ref{lemma quadratic decay}, but from below $+1$. Since for phase-fields which stay close enough to $\pm 1$, there is no significant difference between $u = 1 + (u-1)$ and $\bar u = 1 - (u-1)$, it is clear that such an estimate should hold. The proof of \cite[Proposition 3.5]{roger:2006ta} can be sharpened to show that
\begin{equation}\label{substitute equation}
\mu_\eps(B_r\cap \{1-\tau \leq |u_\eps|\leq 1\} \leq \tau\,\mu_\eps(B_{2r}\cap \{|u_\eps|\leq 1-\tau\}) + \eps^2\,\alpha_\eps(B_{2r}) + o(\eps^2),
\end{equation}
and thus in our case $\mu_\eps(B_r) = o(\eps^2)$. Details will be given in a future publication. Since we know that 
\[
\int_R\frac1\eps\,W'(u_\eps)^2\dx \leq C\,\eps^2,
\]
we can pick a ring
\[
R_\eps = \{y\in \Omega \:|\: r_\eps < |x-y|< r_\eps+ |\log(\eps)|^{-1}\} \cc R
\]
such that
\[
\eps^{-2}\,\mu_\eps(R_\eps) + \int_{R_\eps}\frac1{\eps^3}\,W'(u_\eps)^2\dx \leq C\,|\log(\eps)|^{-1}.
\]
Then we choose a cut-off function $\eta$ such that $\eta \equiv 1$ inside $B_{r_\eps}(x)$, $\eta\equiv 0$ outside $B_{r_\eps + |\log\eps|\inv}(x)$, $|\nabla \eta|\leq 2\,|\log\eps|$, $|\Delta\eps|\leq C\,|\log \eps|^2$ and define
\[
\hat u_\eps = (1-\eta) \,u_\eps + \eta.
\]
Since $|u_\eps|\geq 1/\sqrt{2}$ on $R$, we can suppose that without loss of generality $u_\eps\geq 1/\sqrt{2}$. It follows directly from $\hat u_\eps\geq u_\eps>0$ that 
\[
\int_R\frac1\eps\,W(\hat u_\eps)\dx \leq \int_R\frac1\eps \,W(u_\eps)\dx
\]
and furthermore $0 = \hat \mu_\eps(B_{r_1}(x)) \leq \mu_\eps(B_{r_1}(x))$. Finally, we note that
\[
\int_R\frac\eps2\,|\nabla \hat u_\eps|^2 \dx \leq \eps \int_R|\nabla u_\eps|^2\,(1-\eta)^2  + (u_\eps-1)^2\,|\nabla \eta|^2\dx = O(\eps^2\,|\log\eps|^2)
\]
due to \eqref{substitute equation}, so in particular that $\hat\mu_\eps\wto\mu$ and $\lim_{\eps\to 0}S_\eps(\hat u_\eps) = \lim_{\eps\to 0}S_\eps(u_\eps)$. Since $u_\eps\to 1$ in $L^1(B_{2r}(x))$ already before the modification, we do not change the limiting integral either:
\[
\lim_{\eps\to 0} \int_\Omega \hat u_\eps\dx = \lim_{\eps\to 0} \int_\Omega u_\eps\dx.
\]
Hence the last two terms in $\E_\eps$ converge to the same limits as before. Thus it suffices to show that $\liminf_{\eps\to 0}\W_\eps(\hat u_\eps)< \liminf_{\eps\to 0}\W_\eps(u_\eps)$ to see that
\[
\liminf_{\eps\to 0}\E_\eps(\hat u_\eps) < \liminf_{\eps\to 0}\E_\eps(u_\eps),
\]
which means that $u_\eps$ cannot be a minimiser of $\E_\eps$ for some small $\eps>0$. This is the contradiction we are looking for. So calculate
\begin{align*}
\hat \alpha_\eps&(B_{r_\eps+|\log\eps|\inv}(x)) = \hat\alpha_\eps(R_\eps)\\
	&= \frac1{c_0\,\eps} \int_{R_\eps} \left(\eps\,\Delta \hat u_\eps - \frac1\eps\,W'(\hat u_\eps)\right)^2\dx\\
	&\leq \frac{1+\delta}{c_0\,\eps} \int_{R_\eps} \left(\eps\,\Delta u_\eps - \frac1\eps\,W'(u_\eps)\right)^2(1-\eta)^2\dx\\
		 &\qquad  + \left(1 + \frac1\delta\right) \frac1{c_0\,\eps} \int_{R_\eps} \left(-2\eps \,\langle \nabla \eta, \nabla u_\eps\rangle + \eps\,(1-u_\eps)\,\Delta\eta + \frac1\eps[W'(u_\eps)(1-\eta) - W'(\hat u_\eps)]\right)^2\dx\\
	&\leq \frac{1+\delta}{c_0\,\eps} \int_{R_\eps} \left(\eps\,\Delta u_\eps - \frac1\eps\,W'(u_\eps)\right)^2\eta^2\dx
		+\left(1+\frac1\delta\right) \frac{3}{c_0\,\eps^3}\int_{R_\eps} [W'(u_\eps)(1-\eta)  - W'(\hat u_\eps)]^2\dx\\
		 &\qquad \qquad + \left(1 + \frac1\delta\right) \frac{3\,\eps}{c_0} \int_{R_\eps} 4\,\langle \nabla \eta, \nabla u_\eps\rangle^2 + (u_\eps-1)^2\,(\Delta\eta)^2\dx \\
	&\leq (1+\delta)\,\alpha_\eps(R_\eps) +\left(1+\frac1\delta\right) \frac{3}{c_0\,\eps^3}\int_{R_\eps} W'(u_\eps)^2\dx  + \left(1 + \frac1\delta\right)\,C\,\left(\eps^2\,|\log\eps|^2 + \eps^4\,|\log\eps|^4\right)\\
	&\leq  (1+\delta)\,\alpha_\eps(R_\eps) +\left(1+\frac1\delta\right)\frac{3}{c_0\,|\log\eps|}  + \left(1 + \frac1\delta\right)\,C\,\left(\eps^2\,|\log\eps|^2 + \eps^4\,|\log\eps|^4\right)
\end{align*}
for $\delta>0$. Here we used that $\hat u_\eps$ is a convex combination of $u_\eps$ and $1$ pointwise, so that the estimate in the middle integral works. Taking first $\eps\to 0$ and then $\delta\to 0$, it follows that
\[
\hat\alpha(B_{2r}) =\alpha(B_{2r}\setminus B_{r_0}) < \alpha (B_{2r})
\]
where $r_0 = \lim_{\eps\to 0}r_\eps>r$. This implies the contradiction and concludes the proof.
\end{proof}

Cases of independent geometric interest are the formal limits $\lambda=\chi = \infty$ and $\lambda=\infty$, $\chi =0$. The problem becomes more complex, and in the first case, solutions can only exist if $V < \L^n(\Omega)$ and $S>c_0$ for some $c_0$ depending on $V$ through the isoperimetric inequality and the geometry of $\Omega$. These limits can be expressed in phase field models for example in the choice of the function space
\[
X_\eps(\Omega) = \{u\in -1+ W^{2,2}_0(\Omega)\:|\: S_\eps(u) = S\}
\]
or by choosing $\lambda = \lambda_\eps = \eps^{-1}$, and similarly in the first case. Our simple modification clearly does not go through in either scenario, but we believe that the same result should still hold. We will see that uniform convergence still holds for a penalised functional when we add a version of the topological term discussed in \cite{MR3590663}. For simplicity, we restrict ourselves to the case discussed there, but a total integral term could easily be included.

Let us briefly recall the structure of the topological functionals. Take $\phi_1\in C_c( 1/\sqrt{2}, 1)$ to be a non-negative function with connected support, $\phi_1\not\equiv 0$ and $F_1\geq 0$ a continuous function such that $F_1\cdot\phi_1\equiv 0$ and $F_1(\pm 1) >0$. Set $\phi_2(t) = \phi_1(-t)$ and $F_2(t) = F_1(-t)$. Then we have the geodesic distances
\[
d^{F_i(u)}(x,y) = \inf\left\{\int_K F_i(u)\d\H^1\:\big|\:K\text{ connected and }x,y\in K\right\}
\] 
and the topological energies
\[
C_\eps^i(u) = \frac1{\eps^2}\int_{\Omega\times \Omega}\phi_i(u(x))\,\phi_i(u(y))\,d^{F_i(u)}(x,y)\dx\dy
\]
for $i=1,2$.

\begin{lemma}
Let $X = -1 + W^{2,2}_0(\Omega)$, $\sigma\in (0,4)$, $\kappa>3$ and
\[
\E_\eps\colon X\to [0,\infty), \qquad \E_\eps(u) = \W_\eps(u) + \eps^{-\sigma}\,(S_\eps(u)-S)^2 + \eps^{-\kappa}\,\left(C^1_\eps + C^2_\eps\right)(u).
\]
Assume that $u_\eps\in X$ are minimisers of $\E_\eps$ and $u, \alpha, \mu$ as usual. Then $\spt(\alpha)\subset \spt(\mu)$ and $u_\eps\to u$ locally uniformly in $\overline\Omega\setminus\spt(\mu)$.
\end{lemma}

\begin{proof}[Sketch of Proof:]
The proof proceeds in two steps. In the first one, we assume that if $x\in \spt(\mu)$ and $y\in \Omega$ such that $y_\eps\to y$ and $u_\eps(y_\eps) \in \mathrm{supp}(\phi_1)$. Then we deduce that
\[
\liminf_{\eps\to 0} \left(\frac1{\eps}\int_{B_r(x)}\phi_1(u_\eps(x))\dx\right)\left(\frac1{\eps^3}\int_{B_r(y)}\phi_1(u_\eps(y))\dy\right)>0
\]
for all $r>0$ using \cite[Lemma 3.15]{MR3590663} on the first term and infinitesimal H\"older regularity on the second. Repeating the proof of \cite[Theorem 2.4]{MR3590663}, we obtain a contradiction. Having excluded the situation of Example \ref{second example}, we use the same modification as in Lemma \ref{lemma uniform convergence minimisers} on $u_\eps$ in the second step of the proof. Since we know that $|u_\eps|\geq 1/\sqrt{2}$ on the whole ball $B_{r_2}(x)$ rather than just the ring domain $R$, the difference in the diffuse area functional can be controlled to be $o(\eps^\sigma)$ for all $\sigma<2$. Thus the same argument as above goes through.
\end{proof}

\subsection{Hausdorff-Convergence of Manifolds with Bounded Willmore Energy} Last but not least we show how our results can be used to obtain results on the interplay between varifold and Hausdorff convergence using only our PDE techniques and just a bare minimum of geometric measure theory.

\begin{lemma}
Let $M_k$ be a sequence of compact orientable $C^2$-manifolds embedded in $\R^3$ and $V_k$ their induced varifolds. Assume that $V$ is a varifold with mass measure $||V||$ such that
\[
V_k\wto V, \qquad \limsup_{k\to\infty}\left[\W(V_k) + \H^2(M_k)\right] <\infty.
\]
Then $\lim_{k\to\infty}M_k = \spt(V) \cup \{x_1,\dots, x_N\}$ for a finite collection of points $x_i$, $i=1,\dots, N$ in the sense of convergence in Hausdorff distance for every subsequence along which the limit exists. Moreover, if $M_k$ is connected for all $k\in\N$ or $\lim_{k\to\infty}\W(V_k) = \W(V)$, then
\[
\spt(||V||) = \lim_{k\to\infty}M_k.
\] 
\end{lemma} 

\begin{proof}
Due to an improved estimate of L.~Simon \cite[Lemma 1]{MR1650335}, the manifolds have uniformly bounded diameter
\[
\diam(M_k)\leq \frac2\pi\,\sqrt{\H^2(M_k)\,\W(M_k)\,}.
\]
If they go off to infinity, $||V||\equiv 0$ and the Hausdorff limit is empty, so there is nothing to show. Thus we may assume that $M_k \subset B_R(0)$ for a suitably big $R>0$. If we wished to avoid geometric measure theory altogether, we could alternatively assume that the manifolds $M_k$ are a priori all contained in some bounded subset of $\R^3$. 

A simple contradiction shows that $\spt(||V||) \subset \lim_{k\to\infty}M_k$, so only the inverse direction is difficult. This concerns the uniform or essentially uniform convergence of phase fields away from $u_\eps$ which we established using exclusively PDE techniques and no geometric measure theory at all.

As $M_k$ is compact, orientable and embedded, it is the boundary of a set $E_k\subset\R^3$. Since furthermore $M_k\in C^2$, there is a sequence $\eps_k\to 0$ such that the signed distance function $\sdist(\cdot, M_k)$ is $C^2$-smooth on
\[
U_k = \{x\in \R^n\:|\: \dist(x, M_k) < 2\,\sqrt{\eps_k\,}\}
\]
and we can consider the sequence
\[
u_k\colon B_{R+1}(0)\to \R, \qquad u_k(x) = q_{\eps_k}(\sdist(x,M_k)/\eps_k).
\]
By construction, $u_k\in -1 + W^{2,2}_0(B_{R+1}(0))$ and $M_k \equiv \{u_k = 0\}$. If we choose $\eps_k$ sufficiently small, it becomes obvious that 
\[
\lim_{k\to \infty}\mu_k = ||V||, \qquad \lim_{k\to\infty}\alpha_k(\R^3) = \lim_{k\to\infty}\W(M_k).
\]
We can therefore invoke Corollary \ref{corollary hausdorff} (with boundary values) to see that
\[
\lim_{k\to\infty}M_k = \lim_{k\to \infty}\{u_k=0\} = \spt(||V||) \cup \{x_1, \dots, x_N\}
\]
for a finite collection of points $x_1,\dots,x_N\in \ol{B_{R}(0)}$. Now assume additionally that $M_k$ is connected for all $k\in\N$. Then, by standard results on Hausdorff convergence, $\lim_{k\to\infty}M_k$ is connected, so the finite collection of points must be empty.

Last, assume that we have a recovery sequence, i.e.\ $\lim_{k\to\infty} \W(M_k) = \W(V)$. If we choose $\eps_k$ sufficiently small also $\W_{\eps_k}(u_k)\to\W(V)$, thus 
\[
\lim_{k\to\infty}M_k = \lim_{k\to \infty}\{u_k=0\} = \spt(||V||)
\]
as explained in Corollary \ref{corollary recovery sequence}.
\end{proof}

\begin{corollary}
If $M_k$ is connected for all $k\in\N$, then also $\spt(||V||)$ is connected.
\end{corollary}

In particular, we have shown with phase field techniques that the problem of minimising Willmore's energy in the class of connected surfaces arising as the limits of boundaries is well posed in three dimensions. This is by no means immediate since it is not obvious whether thin cylinders connecting bigger pieces of the manifolds $M_k$ can collapse away in measure with bounded Willmore energy while keeping the bigger parts at a positive distance. 

It is easy to give an example $V_k\wto V$ with additional points in the Hausdorff limit. Namely, Take $M_k\equiv M \cup \partial B_{r_k}(x)$ for some $M$, $x\notin M$ and $r_k\to 0$. Then, if $V$ denotes the varifold induced by $M$, we have
\[
V_k\wto V, \qquad \spt(V) = M, \qquad \lim_{k\to \infty}M_k = M \cup\{x\}, \qquad \W(M_k) \equiv \W(M) + 16\pi.
\]
A diffuse analogue of this example will be discussed in Example \ref{second example}.

\section{Counterexamples to Uniform Convergence}\label{section examples}

In this section, we give examples showing that our results are optimal. All constructions are simple perturbations of an optimal interface recovery sequence.
 
\begin{example}\label{introductory example}
Let $E\cc\Omega$ with $\partial E\in C^2$ and denote by $d(x) = \sdist(x,\partial E)$ the signed distance function from $\partial E$. Like in the introduction, we take an optimal interface transition $u^\eps(x) = q_\eps(d(x)/\eps)$ where $q_\eps$ is constant for arguments larger than $\eps^{-1/2}$. Now take $x_0\notin \partial E$, $g\in C_c^\infty(\R^n)$ and set 
\[
u^\eps_g(x) \coloneqq  u^\eps(x) + \eps^\beta\,g(\eps^{-\gamma}(x-x_0)).
\]
For small enough $\eps>0$, we know that $u^\eps\equiv \pm 1$ close to $x_0$, which simplifies the energies of the modified functions. We may assume that $u^\eps\equiv 1$ around $x_0$. If $\gamma>0$ (or $\gamma=0$ and $g$ has sufficiently small support), this implies the following identities.
\begin{align*}
S_\eps(u^\eps_g) &= S_\eps(u^\eps) + \eps^{1+2\beta+(n-2)\gamma}\int_{\R^n}\frac12\,|\nabla g|^2\d x + \eps^{2\beta+n\gamma-1}\int_{\R^n}g^2\,\frac{(2+\eps^{\beta}g)^2}4\dx,\\
\W_\eps(u^\eps_g) &= \W_\eps(u^\eps) + \int_{\R^n}\frac1\eps\bigg(\eps^{1+\beta-2\gamma}\,\Delta g - \eps^{\beta-1}\,g\,(1+\eps^\beta\,g)\,(2+\eps^\beta\,g)\bigg)^2(\eps^{-\gamma}(x-x_0))\dx.
\end{align*}
In the bending energy, both terms scale differently unless 
\[
1+\beta-2\gamma = \beta-1 \qquad\LRa\qquad \gamma=1.
\]
In this situation, we can simplify the integral to give
\[
\W_\eps(u^\eps_\gamma) \approx \W_\eps(u^\eps) + \eps^{2\beta-3+n}\int_{\R^n}(\Delta g - 2\,g)^2\dx
\]
under the assumption that $\beta>0$. For a compactly supported non-zero function $g$ the last term cannot be zero, so we have the heuristic condition
\[
2\beta -3+n\geq 0\quad\LRa\quad \beta\geq \frac{3-n}2
\]
for the energy to remain finite. Conversely, it is easy to see that the energy does remain finite in these cases in both $n=2$ and $n=3$ dimensions. Setting $\beta=0$ shows that $u_\eps$ need not converge uniformly to $\pm 1$ away from the interface in three dimensions. In two dimensions, setting $\beta=1/2$ shows that we cannot obtain a convergence rate better than $\sqrt{\eps}$. 

Note in particular that we have $\mu^\eps_g(B_r(x)) = O(\eps^2)$ if $\alpha_g$ has an atom at $x$ and $\mu^\eps_g(B_r(x)) = o(\eps^2)$ otherwise, both in two and three dimensions. In three dimensions, we can consider the function
\[
e\colon [0,1)\to(0,\infty), \quad e(\theta)\coloneqq  \inf\big\{ [\W_1 + S_1](u)\:|\:u\in 1+ W^{2,2}_0(B_1(0)), \: u(0) = \theta\big\}.
\]
It is positive and monotone decreasing and satisfies $\lim_{\theta\to 1}f(\theta) = 0$, so we can take a sequence $\theta_n\to 1$ such that
\[
\sum_{n=1}^\infty f(\theta_n)<\infty.
\]
Then we take corresponding minimisers $g_n$, a dense subsequence $x_n$ in $\Omega\setminus \partial E$ and define 
\[
\eps_n= \min\left\{\min_{1\leq i\neq j\leq n} \frac{|x_i-x_j|}2, \:\left({\min_{1\leq i\leq n} \dist(x_i, \partial E)/2\,}\right)^2\right\}, \quad u_n(x) = \begin{cases}\pm g_n((x-x_i)/\eps_n) &\text{ in }B_{\eps_n}(x_n)\\ u(x) &\text{else}
\end{cases}
\]
with the choice of sign for $g_n$ such that the function is continuous. Then $\eps_n\to 0$, $u_n\to u$ in $L^1(\Omega)$ and essentially uniformly, but there exists no open set $\Omega'\cc \Omega\setminus\partial E$ such that $u_n\to u$ uniformly on $\Omega'$.
\end{example}

The next example gives a different modification in three dimensions only. It shows that in $\R^3$, uniform convergence away from the interface may fail even if the discrepancy measures $|\xi_\eps|$ vanish faster than polynomially in $\eps$. Another implication is that there is no guaranteed rate of convergence for $\mu_\eps(\Omega')\to 0$ for $\Omega'\cc\Omega\setminus\spt(\mu)$. 

\begin{example}\label{second example}
Consider a set $\Omega'\cc\Omega\setminus\spt(\mu)$ and $x_0\in\Omega'$. Assume that $u^\eps$ is an optimal profile type recovery sequence, or at least that $u^\eps$ is constant on $\Omega'$. Let $r>0$ such that $B_r(x_0)\cc \Omega'$ and $\eps^{3/4} < r_\eps< r/2$. Then the functions
\[
\bar u_\eps(x) = \begin{cases}u^\eps(x)& x\notin B_r(x_0)\\ \pm q_\eps(\sdist(x, \partial B_{r_\eps}(x_0))) & x\in {B_r(x_0)}\end{cases}
\]
are $C^2$-smooth (if the sign of the optimal profile is chosen correctly and $q_\eps$ is constant beyond $\sqrt{\eps}$). This is a recovery-type sequence for $\partial E$ with an additional interface at spheres $\partial B_{r_\eps}(x_0)$ and can easily be seen to satisfy
\[
\bar\mu_\eps\wto \mu, \qquad \bar\mu_\eps(\Omega') \approx \frac{4\pi\,r_\eps^2}3, \qquad \alpha_\eps(\Omega') \approx 16\pi
\]
since spheres of any radius have Willmore energy $16\pi$ in three dimensions with our normalisation of the Willmore functional. As $r_\eps$ may go to zero arbitrarily slowly, so can $\mu_\eps(\Omega')$. Due to the optimal interface construction, $|\xi_\eps|(\Omega')\leq C\,\eps^\gamma$ for all $\gamma\in(0,\infty)$. 
\end{example}

This shows that no penalisation of the discrepancy measures can enforce uniform convergence away from the interface in three dimensions. In two dimensions, this does not work since small circles have large elastica energy while the Willmore functional on surfaces in $\R^3$ is scaling invariant.

We have investigated similar problems in a previous article \cite{MR3590663} with a different focus and under the additional assumption that 
\begin{equation}\label{equation boundary conditions}
u_\eps\in W^{2,2}_{loc}(\R^n)\text{ and }u_\eps\equiv -1\text{ outside }\Omega,
\end{equation}
although essentially uniform convergence could not be established. By extension theorems, these conditions are equivalent to $u_\eps \in -1 +  W^{2,2}_0(\Omega)$ and thus implied if $\partial \Omega\in C^2$ and $u_\eps\equiv -1$, $\partial_\nu u_\eps\equiv 0$ on $\partial\Omega$. This is a suitable model for some problems in biological applications where we are interested in minimising $\W$ among a class of connected surfaces confined to the container $\Omega$. If we assume boundary conditions \eqref{equation boundary conditions} or periodic boundary conditions, the results may be sharpened as follows:

\begin{enumerate}
\item $u_\eps$ is bounded on $\Omega$ and $u_\eps\to u$ in $L^p(\Omega)$ for all $p<\infty$,

\item it is H\"older continuous with constants as above on every ball $B_\eps(x)\cap \Omega$ for $x\in\overline\Omega$,

\item we may replace ``finitely many points in $\Omega$ (or $\Omega\setminus \spt(\mu)$)" with ``finitely many points in $\overline\Omega$ (or $\overline\Omega\setminus\spt(\mu)$)",

\item the Hausdorff limit is $K=\spt(\mu)$ if $\alpha$ has no atoms outside $\spt(\mu)$, and $K=\spt(\mu)\cup\{x_1,\dots, x_N\}$ in general.
\end{enumerate}
 
In this article, all regularity and convergence results we obtained were valid only in the interior of $\Omega$, away from the boundary. The treatment of boundary values is different, since their blow-ups are at most governed by solutions to the stationary Allen-Cahn equation on half-space, which gives us much less control than the same equation on the whole space. For example, the Modica gradient bound clearly fails.

In certain situations, one may be tempted to prescribe boundary conditions $u_\eps \equiv +1$ on $\Gamma_+\subset\partial\Omega$ and $u_\eps\equiv -1$ on $\Gamma_-\subset\partial\Omega$ and leave $u_\eps$ free on $\partial\Omega\setminus (\Gamma_+\cup \Gamma_-)$ to make an optimal transition. However (for simple geometries), it can be seen that the infimum energy of functionals of the type
\[
\E_\eps = \W_\eps + \eps^{-2\sigma}\,(S_\eps-S)^2
\]
under these boundary conditions is identically zero, where minimising sequence is given by the superposition of optimal profiles for a minimal surface spanning a suitable boundary curve and a solution to the stationary Allen-Cahn equation on half space with suitable boundary values and rescaled, inserted at a convex boundary point. The super-position still has almost zero Willmore energy and the boundary modification can be scaled precisely to create an atom of $\mu$ of the right size. More explicit constructions (and boundary values that prevent these pathological examples) will be presented in \cite{DW_bound}.

\section*{Acknowledgements}
PWD and SW would like to thank M.~R\"oger (Dortmund) for inspiring discussions. PWD acknowledges financial support from the German Scholars Organization / Carl Zeiss Stiftung via their Wissenschaftler-R\"uckkehrerprogramm.
SW would like to thank Durham University for financial support through a Durham Doctoral Studentship and Y.~Tonegawa (Tokyo) for helpful conversations.

\end{document}